\documentclass[12pt]{amsart}

\usepackage{latexsym, amssymb, amsmath, amscd}
 \usepackage{eucal}

\newcommand{\C}{{\mathbb C}}
\newcommand{\Reynolds}{{\mathcal R}}
\newcommand{\E}{{\mathbb E}}
\newcommand{\EE}{{\mathcal E}}
\newcommand{\FF}{{\mathcal F}}

\newcommand{\Z}{{\mathbb Z}}
\newcommand{\N}{{\mathbb N}}
\newcommand{\SL}{\operatorname{SL}}
\newcommand{\GL}{\operatorname{GL}}
\newcommand{\Q}{{\mathbb Q}}

\newcommand{\Ker}{\operatorname{Ker}}

\theoremstyle{theorem}
\newtheorem{theorem}{Theorem}[section]
\newtheorem{corollary}[theorem]{Corollary}
\newtheorem{conjecture}[theorem]{Conjecture}
\newtheorem{proposition}[theorem]{Proposition}

\newtheorem{Remark}[theorem]{Remark}
\theoremstyle{definition}
\newtheorem{definition}[theorem]{Definition}

\newtheorem{remark}[theorem]{Remark}

\newcommand{\R}{{\mathbb R}}

\title[Gaussian Moments Conjecture]{The Gaussian Moments Conjecture\\ and the Jacobian Conjecture}
 \author{Harm Derksen, Arno Van den Essen and Wenhua Zhao}      
 \date{\today}
\address{H. Derksen, Department of Mathematics, University of Michigan,  USA. Email: hderksen@umich.edu}
\address{A. van den Essen, Department of Mathematics, Radboud University Nijmegen, The Netherlands. Email: A.vandenEssen@math.ru.nl}
\address{W. Zhao, Department of Mathematics, Illinois State University, Normal, IL 61761. Email: wzhao@ilstu.edu}

\begin{document}

\begin{abstract}
We first propose what we call the {\it Gaussian Moments Conjecture}.  
%{\it let $X=(X_1,\dots,X_n)$ be a random vector with a multi-variate normal distribution and $P(x_1,\dots,x_n)\in \C[x_1,\dots,x_n]$. 
%Assume that the moments $\E(P(X)^m)=0$ for all $m\geq 1$, then for every 
%$Q(x_1,\dots,x_n)\in \C[x_1,\dots,x_n]$ we have $\E(P(X)^mQ(X))=0$ 
%when $m\gg 0$.} 
We then show that the Jacobian Conjecture follows from the Gaussian Moments Conjecture. Note that the the Gaussian Moments Conjecture is a special 
case of (\cite[Conjecture $3.2$]{Zhao2}). The latter conjecture was  
referred as {\it Moment Vanishing Conjecture} in (\cite[Conjecture A]{FPYZ}) and {\it Integral Conjecture} in \cite[Conjecture $3.1$]{EZ} (for the one-dimensional case).  
We also give a counter-example to 
show that (\cite[Conjecture 3.2]{Zhao2}) 
fails in general for polynomials 
in more than two variables. 
\end{abstract}

   

%60E05  Distributions: general theory

%14R15, 13N10: Secondary: 13A99, 13F20

%16N40 Nil and nilpotent radicals, sets, ideals, rings
%
%16D70 Structure and classification (except as in 16Gxx), direct sum decomposition, cancellation
%
%16D99 None of the above, but in this section

%16S34 Group rings [see also 20C05, 20C07], Laurent polynomial rings

%13N10: Rings of differential operators and their modules 
%14R15: Jacobian problem.
%33C45 Orthogonal polynomials and functions of hypergeometric type (Jacobi, Laguerre, Hermite, Askey scheme, etc.) [see also 42C05 for general orthogonal polynomials and functions]. 
%32C38 Sheaves of differential operators and their modules, $D$-modules. 
%32W99: Differential operators in several 
%    complex variables/None of the above, but in this section. 

\thanks{The first author was partially supported by NSF grant DMS $1302032$ and the third author   partially  
by the Simons Foundation Grant-$278638$.}

 \bibliographystyle{alpha}
    \maketitle

% \tableofcontents

\renewcommand{\theequation}{\thesection.\arabic{equation}}
\setcounter{equation}{0}
\setcounter{section}{0}

\section{\bf Introduction}
 For a random variable $X$ we denote its expected value by $\E(X)$. Suppose that $X=(X_1,\dots,X_n)$ is a random vector with a multi-variate normal distribution.
 We make the following conjecture:

\begin{conjecture}[Gaussian Moments Conjecture $\operatorname{\bf GMC}(n)$]
Suppose that $P(x_1,\dots,x_n)\in \C[x_1,\dots,x_n]$ is a complex-valued polynomial such that the moments $\E(P(X)^m)$ 
are equal to $0$ for all $m\geq 1$. Then for every polynomial $Q(x_1,\dots,x_n)\in \C[x_1,\dots,x_n]$ we have $\E(P(X)^mQ(X))=0$ for $m\gg 0$.
\end{conjecture}

By using translations and linear maps, we can normalize the random vector $X$ such that $X_1,\dots,X_n$ are independent, with mean 0 and variance 1.

The {\it Gaussian Moments Conjecture} is a special case of \cite[Conjecture $3.2$]{Zhao2}. Furthermore, because of Proposition $3.3$  
and relation $(3.2)$ in \cite{Zhao2}, the {\it Gaussian Moments Conjecture}  
is the special case of \cite[Conjecture $3.1$]{Zhao2} for Hermite polynomials. 
Note that (\cite[Conjecture 3.2]{Zhao2}) was later referred as {\it Moment Vanishing Conjecture} in (\cite[Conjecture A]{FPYZ}), and {\it Integral Conjecture} in \cite[Conjecture $3.1$]{EZ} (for one-dimensional case). Unfortunately, this conjecture is false in general, as can be seen from the following  

\begin{proposition}\label{theo:MVCis False}
Let $B=\{(x, y)\in \R^2 \,|\,y\geq 0,\, x^2+y^2\leq 1\}$, $P(x,y)=(x+iy)^2$ and
$Q(x,y)=x+iy$. Then $\int_B P(x,y)^m\,dxdy=0$ for all $m\geq 1$, but $\int_B  Q(x,y)P(x,y)^m\,dxdy\neq 0$ for all $m\geq 1$.
\end{proposition}

\begin{proof} For each $m\ge 1$, by using the polar coordinates $(r, \theta)$ we have 
$$
 \int_B P(x,y)^m\,dxdy=\int_{0}^1\int_{0}^{\pi} r^{2m}e^{2mi\theta}r\,dr d\theta=0;
$$
\begin{align*}
 \int_B Q(x,y)P(x,y)^m\,dxdy&=\int_{0}^1\int_{0}^{\pi}r^{2m+1}e^{(2m+1)i\theta}r\,dr  d\theta\\
&=\frac{2i}{(2m+3)(2m+1)}\neq 0.
\end{align*}
\end{proof}

\begin{remark}
Note that Conjecture $3.2$ 
in \cite{Zhao2} is still open for univariate polynomials. It is also open for 
the (whole) disks or squares centered at the origin for polynomials in two variables.   
\end{remark}

\begin{remark}
 The function $X_1^2+X_2^2$ has an exponential distribution and more generally, $X_1^2+\cdots+X_{2k}^2$ has a $\chi^2$ distribution. So, if the Gaussian Moments Conjecture is true for all $n\ge 1$,
then the conjecture is also true when we replace the Gaussian distributions by exponential or $\chi^2$ distributions.
The  Moments Conjecture  for exponential distributions is equivalent to~\cite[Conjecture 4.1]{EWZ}, which is a weaker form of the  Factorial Conjecture (\cite[Conjecture 4.2]{EWZ}).
\end{remark}

One of the main open conjectures in affine algebraic geometry is the notorious Jacobian Conjecture, 
which was first proposed by O. H. Keller \cite{K} in $1939$. See also \cite{BCW} and \cite{E1}. 

\begin{conjecture}[Jacobian Conjecture ${\bf JC}(n)$] If $F:\C^n\to \C^n$ is a polynomial map that is locally invertible, then it is globally invertible.
\end{conjecture}

The main result of this paper is:

\begin{theorem}\label{theo:MCisJC}
If ${\bf GMC}(n)$ is true for all $n\ge 1$, then  ${\bf JC}(n)$ is true for all 
$n\ge 1$.
\end{theorem}

{\bf Acknowledgment} The main results of the paper were obtained when the authors organized and attended the two-week International Short-School/Conference on Affine Algebraic Geometry and the Jacobian Conjecture, which was supported by and hold at Chern Institute of Mathematics, Tianjin, China from July 14-25, 2014. 
The authors are very grateful to the institute for supports 
and hospitality.

\renewcommand{\theequation}{\thesection.\arabic{equation}}
\setcounter{equation}{0}

\section{\bf Background}
Suppose that $A$ is a unital commutative $\C$-algebra.
%\begin{definition}
%The {\em radical} of a cone $V\subseteq A$ is 
%$$
%\rr(V)=\{f\in A\mid f^m\in V\mbox{ for $m\gg 0$}\}.
%$$
%and the {\em strong radical} of $V$ is
%$$
%\srr(V)=\{f\in A\mid \forall g\in A\ gf^m\in V\mbox{ for $m\gg 0$}\}.
%$$
%\end{definition}
%It is clear that $\srr(V)\subseteq \rr(V)$, $\rr(\rr(V))=\rr(V)$ and $\rr(\srr(V))=\srr(V)$.
\begin{definition} 
A {\em Mathieu-Zhao space} (or {\em MZ space}) is a $\C$-linear subspace $V\subseteq A$ with the property
that $f^m\in V$ for all $m\geq 1$ implies that for every $g\in A$,  $f^mg\in V$ for $m\gg 0$.
\end{definition}

Observe that in this definition we have changed the name {\it Mathieu subspace}, which was introduced by the third author
in \cite{Zhao2, Zhao3}, into {\it Mathieu-Zhao space} or {\it MZ space}. This follows a suggestion of the second author in
\cite{E}. For some more general studies of this new notion, see \cite{Zhao3}.

With the definition above we can now reformulate our main conjecture as follows. 

\begin{conjecture}[${\bf GMC}(n)$, reformulation]
The subspace
$$
\{P(x_1,\dots,x_n)\in \C[x_1,\dots,x_n]\mid \E(P(X_1,\dots,X_n))=0\}
$$
is an MZ space of\, $\C[x_1, \dots, x_n]$.
\end{conjecture}

%\begin{lemma}
%A subspace $V\subseteq A$ is a MZ space if and only if $\rr(V)=\srr(V)$.
%\end{lemma}
%\begin{proof}
%It is clear that if $\rr(V)=\srr(V)$ then $V$ is a MZ space. Suppose that $V\subseteq A$ is a MZ space.
%Assume that $f\in \rr(V)$. Then there exists a $k\geq 1$ such that $f^m\in V$ for $m\geq k$.
%So we have $(f^k)^m\in V$ for $m\geq 1$.  Now $f^k\in \srr(V)$ because $V$ is a MZ space.
%It follows that $f\in \rr(\srr(V))=\srr(V)$. This proves that $\rr(V)=\srr(V)$.
%\end{proof}

Suppose that $G$ is a complex reductive algebraic group acting  regularly on an affine variety $Z$. Then $G$ also acts
on the  ring $\C[Z]$ of polynomial functions on $Z$. Let $K\subseteq G$ be a maximal compact subgroup. Then $K$ is Zariski dense in $G$.
The Reynolds operator $\Reynolds_Z:\C[Z]\to \C$ is the averaging operator:
$$
\Reynolds_Z(f)=\int_{g\in K} g\cdot  f\,d\mu.
$$
where $d\mu$ is the Haar measure on $K$, normalized such that $\int_K d\mu=1$.

\begin{conjecture}[Mathieu Conjecture $\operatorname{\bf MC}(Z)$]
The kernel $\Ker(\Reynolds_Z)$ of the Reynolds operator 
is an MZ space of\, $\C[Z]$.
\end{conjecture}

This conjecture is equivalent to the conjecture $C(\C[Z])$ of \cite{Mathieu} (see \cite[Corollary 1.3]{Mathieu}).
The group $G$ acts on its own coordinate ring, and $\operatorname{\bf MC}(G)$ implies $\operatorname{\bf MC}(Z)$ (\cite[Corollary 1.7]{Mathieu}).
The following theorem was proven in \cite[Theorem 5.5]{Mathieu}:

\begin{theorem}[Mathieu]
If $\operatorname{\bf MC}(\SL_n(\C)/\GL_{n-1}(\C))$ is true for all $n\ge 1$, then ${\bf JC}(n)$ is true for all $n\ge 1$.
\end{theorem}

For later purposes, here we also point out that J. Duistermaat and
W. van der Kallen \cite{DK} in $1998$ had proved the Mathieu conjecture for
the case of tori, which can be re-stated in terms of MZ spaces as follows.

\begin{theorem}[Duistermaat and van der Kallen] \label{DKThm}
 Let $x = (x_1, x_2,..., x_n)$ be $n$ commutative free variables and $M$ the subspace of the Laurent polynomial algebra $\C[x_1^{-1},\dots,x_n^{-1},x_1,\dots,x_n]$ consisting of the Laurent polynomials with no constant term. Then $M$ is an MZ space of $\C[x_1^{-1},\dots,x_n^{-1},x_1,\dots,x_n]$.
\end{theorem}

Let $\partial_i=\frac{\partial}{\partial z_i}$ be the partial derivative with respect to $z_i$. Define
$$\EE_n:\C[w,z]=\C[w_1,\dots,w_n,z_1,\dots,z_n]\to \C[z]$$
such that 
$$
\EE_n\big(P(w)Q(z)\big)=P(\partial)Q(z)\in \C[z].
$$

Zhao made the following conjecture in~\cite{Zhao}:

\begin{conjecture}[Special Image Conjecture ${\bf SIC}(n)$]
  $\Ker(\EE_n)$ is an MZ space of\, $\C[w,z]$.
\end{conjecture}

Zhao proved the following result (\cite[Theorem 3.6, Theorem 3.7]{Zhao}):

\begin{theorem}[Zhao]\label{theo:Zhao}
If\, ${\bf SIC}(n)$ is true for all $n\ge 1$, then ${\bf JC}(n)$ is true for 
all $n\ge 1$.
\end{theorem}

\renewcommand{\theequation}{\thesection.\arabic{equation}}
\setcounter{equation}{0}

\section{\bf Reduction of the Jacobian Conjecture to the Gaussian Moments Conjecture}

We define the linear map 
$$\FF_n:\C[w,z]=\C[w_1,\dots,w_n,z_1,\dots,z_n]\to \C
$$
by setting 
$$
\FF_n(P)=\EE_n(P)\mid_{z=0}.
$$

For $\alpha=(\alpha_1,\dots,\alpha_n)\in \N^n$, set $z^\alpha=z_1^{\alpha_1}\cdots z_n^{\alpha_n}$ and $\alpha!=\alpha_1!\alpha_2!\cdots\alpha_n!$. Then we have
$$
\FF_n(w^\alpha z^\beta)=\left\{
\begin{array}{ll}
\alpha! & \mbox{if $\alpha=\beta$;}\\
0 & \mbox{if $\alpha\neq \beta$.}
\end{array}
\right.
$$

\begin{proposition}\label{prop:EisF}
If $\Ker(\FF_n)$ is an MZ space of $\C[w,z]$, then $\Ker(\EE_n)$ is an MZ space of $\C[w,z]$, i.e. ${\bf SIC}(n)$ is true.
\end{proposition}

\begin{proof}
Assume that $P^m\in \Ker(\EE_n)$ for $m\geq 1$. Then 
for each $\alpha\in \C^n$ we have 
$$
\EE_n(P^m(w,z))\mid_{z=\alpha}=\EE_n(P^m(w,z+\alpha))\mid_{z=0}=\FF_n(P^m(w,z+\alpha))=0.
$$

Hence $P^m(w,z+\alpha)\in \Ker(\FF_n)$ for all $m\geq 1$. Since $\Ker(\FF_n)$ is an MZ space of $\C[w, z]$, for any 
$Q\in \C[w,z]$ and $\alpha\in \C^n$ we have
 $Q(w,z+\alpha)P(w,z+\alpha)^m\in \ker(\FF_n)$ for all $m\gg 0$. 
Therefore, for all $m\gg 0$ we have
 $$
 \EE_n(Q(w,z)P(w,z)^m)\mid_{z=\alpha}=\FF_n(Q(w,z+\alpha)P(w,z+\alpha)^m)=0.
 $$

Define $Z_N\subseteq \C^n$ to be the zero
set of all $\EE_n(Q(w,z)P(w,z)^m)$ with $m\geq N$. Clearly, $Z_N$ is Zariski closed for all $N$, and $\bigcup_{N=1}^\infty Z_N=\C^n$.
It follows that $Z_N=\C^n$ for some integer $N$, because a countable union of Zariski closed proper subsets cannot be the whole affine space.
So for $m\geq N$, $\EE_n(Q(w,z)P(w,z)^m)$ is the zero function.
\end{proof}

\begin{proposition}\label{prop:GMCisSIC}
If ${\bf GMC}(2n)$ is true, then $\Ker(\FF_n)$ is an MZ space of $\C[w,z]$.
\end{proposition}

\begin{proof}
Let $X_1,\dots, X_n,$ $Y_1,\dots,Y_n$ are $2n$ independent random variables with the normal distribution and with mean $0$ and variance $1$.
Define complex-valued random variables $W_j$, $Z_j$ and real-valued random variables $R_j$ and $T_j$ by
$$
W_j=\frac{X_j-Y_ji}{\sqrt{2}}=R_je^{-iT_j}\mbox{ and } Z_j=\frac{X_j+Y_ji}{\sqrt{2}}=R_je^{iT_j}.
$$
Then $R_1,\dots,R_n,T_1,\dots,T_n$ are independent, and for every $1\le j\le n$, 
$R_j^2$ has an exponential distribution with mean $1$ and $\E(R_j^{2k})=k!$.
Now consider 
$$
\E(W^\alpha Z^\beta)=\E(R^{\alpha+\beta}e^{i\sum_{j}(\beta_j-\alpha_j)T_j})=\prod_{j=1}^n\Big( \E(R^{\alpha_j+\beta_j})\E(e^{i(\beta_j-\alpha_j)T_j})\Big).
$$
If $\beta\neq \alpha$, then $\beta_j\neq \alpha_j$ for some $j$, whence  
$\E(e^{i(\beta_j-\alpha_j)T_j})=0$ and $\E(W^\alpha Z^\beta)=0$. If $\alpha=\beta$,  then we have
$$
\E(W^\alpha Z^\alpha)=\E(R^{2\alpha})=\prod_j \E(R_j^{2\alpha_j})=\prod_j\alpha_j!=\alpha!
$$
It follows that $\E(W^\alpha Z^\beta)=\FF_n(w^\alpha z^\beta)$ for all $\alpha,\beta\in \N^n$. By linearity, we get
$\E(Q(W,Z))={\mathcal F}_n(Q(w,z))$ for every polynomial $Q(w,z)\in \C[w,z]$. It follows readily from ${\bf GMC}(2n)$ that
$\Ker\FF_n$ is an MZ space of $\C[w, z]$. 
\end{proof}

Now we can prove our main result Theorem~\ref{theo:MCisJC}:    

\begin{proof}[Proof of Theorem~\ref{theo:MCisJC}]
It follows directly from Proposition~\ref{prop:EisF}, Proposition~\ref{prop:GMCisSIC} and Theorem~\ref{theo:Zhao}.
\end{proof}

\renewcommand{\theequation}{\thesection.\arabic{equation}}
\setcounter{equation}{0}

\section{\bf Some Special Cases of the Gaussian Moments Conjecture}

We view $\C[x_1,\dots,x_n]$ as the coordinate ring of $V\cong \C^n$, where $V$ is viewed as the standard representation  of ${\rm O}(n)$.
\begin{proposition}
For homogeneous polynomials $P(x)$, ${\bf GMC}(n)$ follows from ${\bf MC}(V)$.
\end{proposition}
\begin{proof}
Let $\Phi:\C[x_1,\dots,x_n]\to \C$ be given by $\Phi(P(x))=\E(P(X))$.
Any linear map $\C[x_1,\dots,x_n]_d\to \C$ is determined by an element of $S^d(V)$. Since $\Phi$ is invariant under the action of $O(n)$
it is given by an element of $S^d(V)^{{\rm O}(n)}$. But $S^d(V)^{{\rm O}(n)}$ is at most one dimensional and is spanned by the restriction of the Reynolds operator $\Reynolds_V$.
So up to a constant, $\Phi(P(x)^m)$ is equal to $\Reynolds_V(P(x)^m)$. If $\E(P(X)^m)=0$ for $m\geq 1$,
then $\Reynolds_V(P(X)^m)=0$ for $m\geq 1$. If $Q(x)$ is homogeneous, then $\Reynolds_V(P(x)^mQ(x))=0$ for $m\gg 0$.
So $\E(P(X)^mQ(X))=0$ for $m\gg 0$. If $Q(X)$ is non-homogeneous then  $\E(P(X)^mQ(X))=0$ for $m\gg 0$, because $\E(P(X)Q_d(X))=0$ for $m\gg 0$ for every
homogeneous summand $Q_d(x)$ of $Q(x)$.
\end{proof}

\begin{proposition}\label{Propo4.2}
Suppose that $X$ is a Gaussian Random Variable, and  $P(x)\in \C[x]$ is a univariate polynomial such that 
 $\E(P(X)^m)=0$ for $m\geq 1$, then $P(x)=0$. In particular, 
${\bf GMC}(n)$ is true for $n=1$.
\end{proposition}

\begin{proof} As observed in the beginning of this paper,  ${\bf GMC}(n)$ is a special case of
the Image Conjecture for Hermite polynomials. For $n=1$ the case of Hermite polynomials is proved in Corollary 4.3 of \cite{EZ}.
\end{proof}

For a different proof of ${\bf GMC}(1)$, see Proposition \ref{Propo4.6} and 
Remark \ref{Rmk4.8} of this section. 
 
\begin{proposition}\label{DoubleHgs}
Let $P\in \C[x_1,   ..., x_n, \, y_1,   ..., y_n]$ such that 
for each $1\le k\le n$ $P(x, y)$ as a polynomial 
in $x_k$ and $y_k$ is homogeneous. Then ${\bf GMC}(2n)$ holds for $P$. 
\end{proposition}

\begin{proof} 
For each $1\le k\le n$, let $d_k$ be the degree of $f$ as a polynomial 
in $x_k$ and $y_k$. 

Making the change of variables for $x_i$ and $y_i$ $(1\le i\le n)$:
$$
x_i=r_i\cos \theta_i \quad \text{and} \quad
y_i=r_i\sin \theta_i,
$$
we see that $P=(r_1^{d_1}r_2^{d_2}\cdots r_n^{d_n}) F$ 
for some polynomial $F$ in $\cos\theta_i$ and $\sin\theta_i$ $(1\le i\le n)$, which is independent on $r_i$ $(1\le i\le n)$. 

Let $S^n:=(S^1)^{\times n}$, where $S^1$ is the unit circle in $\C$. 
Denote by $d\mu_n$ the measure of $d\theta_1d\theta_2\cdots d\theta_n$, which is a haar measure of the torus $S_n$.
Then $F$ can be viewed as $S^n$-finite function over the torus $S^n$.
Furthermore, for any $m\ge 1$ we have
\begin{align}\label{DoubleHgs-peq1}
\E(P^m(X, Y))&=\int_{r_1=0}^1 \cdots \int_{r_n=0}^1 
(r_1^{md_1+1}\cdots r_n^{md_n+1})
\big (\int_{S^n} F^m d\mu_n\big) \, dr_1\cdots dr_n\\
&=A_m\int_{S^n} F^m d\mu_n, \nonumber 
\end{align}
for some nonzero constant $A_m$.    

Hence, if $\E(P^m)=0$ when $m\gg 0$, then so is $\int_{S_n} F^m$. 
Since $d\mu_n$ is a Haar measure of the torus $S_n$, applying 
the Duistermaat-van der Kallen Theorem \ref{DKThm} to $F$  
we see that for each polynomial $G$ in 
$\cos\theta_i$ and $\sin\theta_i$ $(1\le i\le n)$, 
we have $\int_{S^n} F^m G\, d\mu_n=0$ when $m\gg 0$. 

Now for each monomial $M(x, y)$ in $x_i$ and $y_i$ $(1\le i\le n)$, 
by Eq.\,(\ref{DoubleHgs-peq1}) with $P^m$ replaced by $P^mM$, 
we see that $\E(P^mM)=0$ when $m\gg 0$. Hence for each polynomial 
$Q(x, y)$, we also have $\E(P^mQ)=0$ when $m\gg 0$. 
Therefore ${\bf GMC}(2n)$ holds for $P$.  
\end{proof}

Since every homogeneous polynomial in two variables 
satisfies the condition of Proposition \ref{DoubleHgs}, we 
immediately have the following 

\begin{corollary}
${\bf GMC}(2)$ holds for all homogeneous polynomials $P$. 
\end{corollary}

By a similar argument as in the proof of Proposition \ref{DoubleHgs}, 
we have also the following case of Conjecture $3.2$ in \cite{Zhao2}:

\begin{corollary}
Let $B$ be the unit disk in $\R^2$ centered at 
the origin with the Lebesgue measure $dxdy$. 
Let $P\in\C[x, y]$ such that $P$ is homogeneous and 
$\int_B P^m\, dxdy=0$ for all $m\gg 0$. Then for every 
$Q\in\C[x, y]$ we have $\int_B P^mQ\, dxdy=0$ 
for all $m\gg 0$.
\end{corollary}

In the rest of this section we point out that some 
results proved in \cite{EWZ} for the Factorial Conjecture 
(\cite[Conjecture 4.2]{EWZ}) 
can also be proved similarly for ${\bf GMC}(n)$.   

First, we give a proof for the following case of 
${\bf GMC}(n)$, which is parallel to 
\cite[Proposition $4.8$]{EWZ}. 

\begin{proposition}\label{Propo4.6-0}
Let $F(x)\in \C[x_1, x_2, ..., x_n]$ such that $F(0)\ne 0$.  
Then $\mathbb E (F^m(X))\ne 0$ for infinitely many $m\ge 1$. 
\end{proposition}

\begin{proof}
Let $\Phi:\C[x_1,\dots,x_n]\to \C$ be given by $\Phi(P(x))=\E(P(X))$. Set $(-1)!!:=1$ and 
$(2k-1)!!:=(2k-1)(2k-3)\cdots 3\cdot 1$ for all $k\ge 1$. Furthermore, 
for each $\alpha=(\alpha_1, \alpha_2, \dots, \alpha_n)\in 2\N$, we set 
$(\alpha-1)!!:=\prod_{i=1}^n (\alpha_i-1)!!$.   
Then for each 
$\alpha\in \N^n$, we have 
\begin{align}\label{P4.6-0-peq1}
\Phi(x^\alpha)=
\begin{cases}
 (\alpha-1)!! &\text{ if } \alpha\in 2\N^n;\\
\quad\,\, 0 &\text{ otherwise.} 
\end{cases}
\end{align}

Now assume that the proposition fails, i.e., 
there exists $N\ge 1$ such that $\Phi(F^m)=0$ 
for all $m\ge N$. Since $F(0)\ne 0$, replacing $F$ by $F/F(0)$  
we may assume $F(0)=1$. Write $F(x)=1- \sum_{i=1}^k c_i x^{\beta_i}$ 
with $c_i\in \C$ and $0\ne \beta_i\in \N^n$ for 
all $1\le i\le k$. 

Note that if $c_i=0$ for all $1\le i\le k$, i.e., $F(x)=1$, the proposition 
obviously holds. So we assume $c_i\ne 0$ for all $1\le i\le k$. 
Replacing $F$ by $F^2$  we may also assume that $0\ne \beta_i\in 2\N$  
for at least one $1\le i\le k$. 

Furthermore, by a reduction due to Mitya Boyarchenko (see the proof of 
\cite[Theorem $4.1$]{FPYZ} or \cite[Remarks $4.5$ and $4.6$]{EWZ}), 
we may also assume that $c_i\in \bar \Q$ for all $1\le i\le k$.   

Let $B=\Z[c_1, c_2, ..., c_k]$ and $p$ 
be an odd prime such that $p\ge N$ 
and $\nu_p(c_i)=0$ for all $1\le i\le k$, where $\nu_p$ denotes an 
extension of the $p$-valuation of $\Z$ to $B$.

Since $p\ge N$ and $F^p\equiv 1-\sum_{i=1}^k c_i^p x^{p\beta_i} 
\,\,(\operatorname{mod } pB)$, 
we have  $\Phi(F^p)=0$ and 
\begin{align}\label{P4.6-0-peq3}
1\equiv \sum_{\substack{1\le i\le k\\ 0\ne \beta_i\in 2\N^n}} c_i^p \,
{(p\beta_i-1)!!} \,\,(\operatorname{mod } pB).
\end{align}
 
Since each $0\ne \beta_i\in 2\N^n$ in the sum above has at 
least one nonzero (and even) component, so $(p\beta_i-1)!!$ is divisible by 
$p$. Then applying $\nu_p$ to Eq.\,(\ref{P4.6-0-peq3}) we get $\nu_p(1)=0$, 
which is a contradiction. 
\end{proof}

The next proposition is parallel to  
\cite[Proposition $4.10$]{EWZ}. 

\begin{proposition}\label{Propo4.6}
Let $F(x)=c_0M_0+ \sum_{i=1}^d c_i M_i$  
with $M_0=x_1^{k_1}\cdots x_n^{k_n}$ such that $k_1\ge 1$ and $k_1\ge k_j$ for all 
$2\le j\le n$; $c_i\in \C$ $(0\le i\le d)$ with $c_0\ne 0$; and $M_i$ $(1\le i\le d)$ 
are monomials in $x$ that are divisible by $x_1^{k_1+1}$. 
Then $\mathbb E (F^m(X))\ne 0$ for infinitely many $m\ge 1$. 
\end{proposition}

\begin{proof} Replacing $F$ by $c_0^{-1}F$ 
we may assume $c_0=1$ and replacing $F$ by $F^2$ we may assume that $k_1$ is an 
even positive integer. Then under these assumptions 
the proof of \cite[Proposition $4.10$]{EWZ} works through similarly 
for the linear functional $\Phi$ of $\C[x_1, \dots, x_n]$ given in Eq.\,(\ref{P4.6-0-peq1}).
\end{proof} 

\begin{Remark}\label{Rmk4.8}  Note that when $n=1$ the conditions of  
Proposition \ref{Propo4.6} hold automatically for all 
nonzero univariate polynomials $F(x)$. Hence ${\bf GMC}(1)$ also follows 
directly from Proposition \ref{Propo4.6}. 
\end{Remark}

%\begin{proposition}
%Let $d\ge 1$ and $P(x)=(\sum_{i=1}^n c_i x_i)^d\in \C[x_1, ..., x_n]$ for some $c_i\in \C$ $(1\le i\le n)$. 
%Assume that $\mathbb E (P^m(X))=0$ for all $m\gg 0$. Then $P=0$. 
%In particular, ${\bf GMC}(n)$ holds for $P(x)$.  
%\end{proposition}
%
%This proposition can be proved similarly as Proposition $4.11$ in 
%\cite{EWZ} if we choose the integer $m$ there to be even, 
%and write the coefficients $c_1^{i_1}c_2^{i_2}\cdots c_n^{i_n}$ 
%there as $(c_1^2)^{i_1/2}(c_2^2)^{i_2/2}\cdots (c_n^2)^{i_n/2}$ 
%(noting that $i_j$ $(1\le k\le n)$ may be assumed to be all even in our case), 
%and then replace $c_i$ $(1\le i\le n)$ by $c_i^2$ for the 
%rest of the proof there. 

\begin{proposition}
Let $d\ge 1$ and $P(x)=\sum_{i=1}^n c_i x_i^d\in \C[x_1, ..., x_n]$ 
for some $c_i\in \C$ $(1\le i\le n)$. 
Assume that $\mathbb E (P^m(X))=0$ for all $m\gg 0$. Then $P=0$. 
In particular, ${\bf GMC}(n)$ holds for $P(x)$.  
\end{proposition}
 
This proposition can be proved similarly as Proposition $4.16$ in 
\cite{EWZ} if we choose the integer $m$ there to be even, and 
the prime $p$ to be $(m+2)d-1$ or $(m+1)d-1$, depending $d$ is odd or even, 
respectively. Note that the components $k_i$'s in the proof of Proposition $4.16$ in 
\cite{EWZ} for our case must be even when $m$ is chosen to be even.    

\renewcommand{\theequation}{\thesection.\arabic{equation}}
\setcounter{equation}{0}

\section{\bf Moment Vanishing Polynomials}

 Let again $X=(X_1,\dots,X_n)$ be a random vector with joint Gaussian distribution.
 For $n\geq 2$, there exist many polynomials $P(x)\in \C[x]$ for which $\E(P(X)^m)=0$ for all $m\geq 1$: if $0$ lies in the  closure of the $O(n)$ orbit of $P(x)$, then $\E(P(x)^m)=0$ for all $m\geq 1$.
 Indeed, if there exists a sequence of orthogonal matrices $A_1,A_2,\dots$
 such that $\lim_{k\to\infty} P(A_k(x))=0$, then we have $\E(P(X))=\lim_{k\to\infty} \E(P(A_k(X)))=\E(\lim_{k\to \infty} P(A_k(X)))=\E(0)=0$.
 A 1-parameter subgroup is a homomorphism $\lambda:\C^\star \to {\rm O}_n(\C)$ of algebraic groups. We can view $\lambda$ as an orthogonal matrix
 with entries in $\C[t,t^{-1}]$. If $P(\lambda(t)(x))$ lies in $t\C[t][x]$, then $\lim_{t\to0}P(\lambda(t)x)=0$ and $0$ lies in the closure of the $O_n(\C)$ orbit of $P(x)$.
Conversely, the Hilbert-Mumford criterion states that if $0$ lies in the ${\rm O}_n(\C)$-orbit closure of $P(x)$, then there exists such a $1$-parameter subgroup 
$\lambda:\C^\star\to {\rm O}_n(\C)$
such that $P(\lambda(t)(x))\in t\C[t][x]$. If $Q(x)\in \C[x]$, then for large $m$, $Q(\lambda(t)(x))P(\lambda(t)x)^m\in t\C[t][x]$ and
$$
\E(Q(X)P(X)^m)=\E(\lim_{t\to 0} Q(\lambda(t)(X))P(\lambda(t)X)=\E(0)=0.
$$
We make the following conjecture:
\begin{conjecture}
If $\E(P(X)^m)=0$ for all $m\geq 1$, then there exists a 1-parameter subgroup $\lambda:\C^\star\to {\rm O}_n(\C)$ such that $P(\lambda(t)(x))\in t\C[t][x]$.
\end{conjecture}

  \end{document}